\documentclass[a4paper,12pt,reqno]{amsart}
\usepackage{amsmath,amsfonts,amssymb}
\usepackage{amsthm}
\usepackage{amstext,amscd,mathabx,mathdots}
\usepackage{graphics,graphicx,epsf}
\usepackage[usenames]{color}
\usepackage{hyperref}
\usepackage{tikz}
\usepackage{xcolor}
\usepackage{color}
\usepackage{float}
\usetikzlibrary{arrows}
\newtheorem{theorem}{Theorem}[section]

\newtheorem{proposition}[theorem]{Proposition}
\newtheorem{lemma}[theorem]{Lemma}
\newtheorem{corollary}[theorem]{Corollary}
\newtheorem{remark}[theorem]{Remark}

\numberwithin{equation}{section}
\title[Fractional powers approach of maximal accretive]{Fractional powers approach of maximal accretive operators for a class of third order abstract Cauchy problems} 

\author[M. Benharrat]{Mohammed Benharrat}
\address[M. Benharrat]{
	Ecole Nationale Polytechnique d'Oran-Maurice Audin (Ex. ENSET d'Oran), 
	BP 1523 Oran-El M'naouar, 31000 Oran, Alg\'{e}rie.}
\email{mohammed.benharrat@enp-oran.dz}
\author[F. D. M. Bezerra]{Flank D. M. Bezerra}
\address[F. D. M. Bezerra]{Departamento de Matem\'atica, Universidade Federal da Para\'iba, 58051-900 Jo\~ao Pessoa PB, Brazil.}
\email{flank@mat.ufpb.br}

\date{\today
\newline
The first author gratefully
acknowledge the financial support from the Laboratory of Fundamental and Applicable Mathematics of Oran (LMFAO) and  the Algerian research project: PRFU, no. C00L03ES310120220003 (D.G.R.S.D.T). The second author gratefully
acknowledge the financial support from CNPq/Brazil grant number 303039/2021-3.}
\begin{document}
\maketitle

\begin{abstract}
In this paper we explore the theory of fractional powers of maximal accretive  operators  to obtain results of existence, regularity and behavior asymptotic of solutions for linear abstract evolution equations of third  order in time.


\end{abstract}

\tableofcontents

\listoffigures

\section{Introduction and preliminaries} 

In this paper we consider the  following abstract linear evolution equation of  third order,
\begin{equation}\label{Eq1}
u'''(t) +Au(t) = 0,\quad t> 0,
\end{equation}
with initial conditions given by
\begin{equation}\label{IC}
u(0)=u_0\in \mathcal{H}^{\frac{2}{3}},\ u''(0)=u_1\in \mathcal{H}^{\frac{1}{3}},\ u'''(0)=u_2\in \mathcal{H}, 
\end{equation}
where $\mathcal{H}$ be a separable Hilbert space and $A: \mathcal{D}(A)\subset \mathcal{H}\to \mathcal{H}$ be a maximal $\omega$-accretive operator, $0\leq \omega\leq \dfrac{\pi}{2} $. Here $u'(t)=\dfrac{du}{dt}(t).$

In this paper we explore the theory of fractional powers of maximal accretive  operators to obtain results of existence, regularity and behavior asymptotic of solutions for  \eqref{Eq1}-\eqref{IC}.

To better present our results in this paper, we introduce first  some notations and terminologies. 
Recall that the numerical range of  an operator $A$ with domain $\mathcal{D}(A)$ in $\mathcal{H}$, is defined by
\begin{equation*}
	W(A):=\{ \left\langle Au, u\right\rangle_\mathcal{H}  : \quad  u\in \mathcal{D}  (A) , \quad \text{  } \left\|u\right\|_\mathcal{H}=1\}.
\end{equation*}

An operator $A$ is said to be accretive if if
its numerical range is contained in the closed right half-plane $\overline{\mathbb{C}_+} := \{ z\in\mathbb{C} : {\rm Re} (z) \geq 0\}$, and $A$ is \textit{maximal accretive}, or \textit{$m$-accretive} for short, if has no proper accretive extensions in $\mathcal{H}$. It is known (see \cite{Kato61}) that
a closed, maximal accretive operator is densely defined, that $A$ is closed
and $m$-accretive if and only if $A^{*}$ is, and also if and only if $-A$ is
the infinitesimal generator of a contraction semigroup $\{\mathcal{T}(t)\}_{t\geq 0}$, that
is, $\Vert\mathcal{T}(t)\Vert_{\mathcal{B}(\mathcal{H})}\leq 1$\footnote{Here, $\mathcal{B}(E)$ denotes the space of bounded linear operators defined in a Banach space $E$ into self endowed with norm
	\[
	\|S\|_{\mathcal{B}(E)}:=\sup_{x\in E,\ x\neq0}\dfrac{\|Sx\|_{E}}{\|x\|_{E}},\ \forall S\in \mathcal{B}(E).
	\]    
}; $A$ is called dissipative, if  $-A$  is accretive. 
For  $\omega \in [0,\pi/2)$, we denote by
$$
\mathcal{S}(\omega):=\left\{z\in\mathbb{C} : |\arg (z)|< \omega \right\},
$$
the open sector of $\mathbb{C}$ with semi-angle $\omega$. An operator $A$  will be said is \emph{$\omega$-accretive}   if 
\begin{equation*}
	\label{sectorial}
	W(A) \subset \overline{\mathcal{S}(\omega)}:=\left\{z\in\mathbb{C} : |\arg (z)|\leq \omega \right\},
\end{equation*}
or, equivalently, $$|{\rm Im} \left(\left\langle Ax, x\right\rangle_\mathcal{H} \right) |\!\le\! \tan(\omega) \, {\rm Re} \left(\left\langle Ax, x\right\rangle_\mathcal{H} \right), \quad \text{for all }  x \in  \mathcal{D}  (A).$$

An $\omega$-accretive  operator  $A$ is called   m-$\omega$-accretive, if it is $m$-accretive. We have $A$ is m-$\omega$-accretive if and only if operators $e^{\pm i\theta} A$ are m-accretive for $\theta=\frac{\pi}{2}-\omega$, $0 < \omega\leq  \pi/2$.
If $\omega=\pi/2$, then we have $\overline{\mathcal{S}(\pi/2)}=\overline{\mathbb{C}_+}$ and m-$\pi/2$-accretivity means m-accretivity. Also, m-$0$-accretive is the nonegative selfadjoint operator. 

The resolvent set of an m-$\omega$-accretive operator $A$ contains the
set $\mathbb{C} \setminus \overline{\mathcal{S}(\omega)}$ and
\[
\|(\lambda I-A)^{-1}\|_{\mathcal{B}(\mathcal{H})} \leq \cfrac{1}{{\rm dist}\left(\lambda,\mathcal{S}(\omega)\right)},
\quad \lambda\in \mathbb{C}\setminus \overline{\mathcal{S}(\omega)}.
\]
Consequently,  the numerical range of  m-$\omega$-accretive operator, $0 \leq \omega\leq  \pi/2$,  has the so-called spectral inclusion property, 
\begin{equation}\label{specinclusion}
	\sigma (A)\subset\overline{ W(A)}\subset \overline{\mathcal{S}(\omega)}.
\end{equation}

It is well  known that  if  $A$  is m-$\omega$-accretive, $0\leq \omega<\pi/2$, then it generates contractive $C_0$-semigroup $\mathcal{T}(t)=e^{-tA}$, $t\ge 0$, and has an holomorphic continuation into the sector $ \overline{\mathcal{S}(\pi/2-\omega)}$, see Theorem IX-1.24 in \cite{Kato}. In this case, we can write $\mathcal{T}(z)=e^{-zA}$, for all $z\in  \overline{\mathcal{S}(\pi/2-\omega)}$; moreover, we have
$$\left\|e^{-zA}\right\|_{\mathcal{B}(\mathcal{H})} \leq 1, \quad z\in \overline{\mathcal{S}(\pi/2-\omega)}. $$


Let $A$ is $m$-accretive and a number  $\alpha\!\in\!(0,1)$; then there exist several equivalent definitions
of a fractional power of the operator $A$. We give one of them: an operator $A^\alpha$ is the closure with $\mathcal{D} (A)$ of the operator
\begin{equation}\label{FRACP}
A^\alpha  =\dfrac{\sin (\pi \alpha )}{\pi} \int_{0}^{\infty}\lambda^{\alpha -1}A(I\lambda+A)^{ -1}dt,	
\end{equation}
see \cite{Balakrishnan60}. The operators $A^{\alpha}$ are m-$(\alpha\pi/2)$-accretive and  $A^\alpha A^\beta x=A^{\alpha +\beta }x$, for all $x\in \mathcal{D} (A^{\alpha +\beta })$, where $\beta>0$ and $\alpha+\beta<1$. Fractional powers have the important multiplicative property
$$
( A^\alpha)^\beta=A^{\alpha \beta }
$$
when  $0<\alpha<1$  and $\beta>0$, \cite[Theorem 4.2]{Martinez88}.
If $\alpha\!\in\! (0,1/2)$, then
$\mathcal{D} (A^{\alpha})= \mathcal{D} (A^{*\alpha})$ holds  in a topological sens, where  the domains of definition are equipped  with corresponding graph norms, see \cite{Kato61}.  Such fractional powers have been defined for a more
general class of linear operators in Banach spaces by several authors (see,
among others, Amann \cite{Amann1995}, Balakrishnan \cite{Balakrishnan60}, Hasse \cite{Hasse2006} and Martinez \cite{Martinez88}).


Denote by $\mathcal{H}^{\alpha}=\mathcal{D}(A^{\alpha})$ for $0\leqslant\alpha\leqslant1$ (taking $A^0:=I$ on $\mathcal{H}^0:=\mathcal{H}$ when $\alpha=0$). Recall that $\mathcal{H}^{\alpha}$ is dense in $\mathcal{H}$  for all $0\leqslant\alpha\leqslant1$, for details see Amann \cite[Theorem 4.6.5]{Amann1995}. The fractional power space $\mathcal{H}^\alpha$ endowed with the norm 
\[
\|\cdot\|_{\mathcal{H}^\alpha}:=\|A^{\alpha} \cdot\|_\mathcal{H}
\] 
is a Banach space.  With this notation, we have $\mathcal{H}^{-\alpha}=(\mathcal{H}^\alpha)'$ for all $\alpha>0$, see Amann \cite{Amann1995}, Sobolevski\u{\i} \cite{Sobolevskii1966} and Triebel \cite{Triebel1978} for the characterization of the negative scale. 

Naturally, one can convert \eqref{Eq1}-\eqref{IC} to the first order Cauchy problem. However, the linearized operator   matrix is in general not closed.  Even if it is closed, it does not generate a strongly continuous semigroup unless under very restrictive conditions. This forces us to seek a suitable phase space to change the awkward situation. The basic idea is to calculate the fractional powers of order $\alpha\in(0,1)$ of  the  linearized operator and we are interested in finding for which values of  $\alpha\in [0,1]$ the fractional first order system is well-posed in some sense. For such $\alpha$, we give an  explicit representation formula of the solution of the  following associated  evolution equations of third order in time, 
\begin{equation}\label{Eq2int}
	u'''(t)+3\varUpsilon_0^\alpha A^{\frac{\alpha}{3}}u''(t)+ 3\varUpsilon_0^\alpha A^{\frac{2\alpha}{3}}u''(t)+A^{\alpha}u (t)= 0,\quad t> 0,
\end{equation}
with initial conditions given by
\begin{equation}\label{IC2int}
	u(0)=\varphi\in \mathcal{H}^{\frac{2}{3}},\ u'(0)= \psi \in \mathcal{H}^{\frac{2-\alpha}{3}},\ u''(0)=\xi\in \mathcal{H}^{\frac{2-\alpha}{3}}, 
\end{equation}
where $\varUpsilon_0^\alpha=2\cos\dfrac{2\pi\alpha}{3} +1$,  $0<\alpha< \alpha^*$ and $\alpha^*= \dfrac{3\pi}{4\pi+2\omega}$.

Fractional powers approach of operators for Cauchy problems associated with evolutionary equations has been studied in some situations, in the sense of existence, regularity and behavior asymptotic of solutions for theses problems, see e.g., \cite{Belluzi2022,BezerraCCN2017, BezerraCN2020, BezerraN2019,Chen2023,Cholewa2018, Dlotko2018,Wu2003}, and references therein.

We extend the work done in \cite{BezerraS2020}, where  $A$ is assumed a self-adjoint positive definite operator ($\omega=0$). They proved that  \eqref{Eq2int}-\eqref{IC2int} admits unique solution, without given explicitly,    for  all  $0<\alpha< \frac{3}{4}$. This means that    $\alpha^*=\frac{3}{4}$ when $\omega=0$. The analogy of this idea to the  general case, up to now is still not obvious. Also some models of continuous mechanics are reduced to differential equation  with sectorial
operators, see \cite{Adamjan2002, Thompson72} and references therein. In this cases methods,
developed for self-adjoint operators, cannot be applied. Here, we continue the analysis done in the previously cited works.  
  
  Summaries of our main results and the structure of the paper are as follows. In Section \ref{Sec2}  we define the linearized  operator, denoted by $\mathbb{A}$, associated to the Cauchy problem \eqref{Eq1}-\eqref{IC} in an appropriate phase space. We study the spectral properties of $\mathbb{A}$. We prove that $\mathbb{A}$ cannot generate a strongly continuous  semigroup of any kind but it is of  non-negative type (and not necessarily accretive) operator. Moreover, under Balakrishnan formula, see \cite{Balakrishnan60}, we explicitly calculate the fractional powers of the operator $\mathbb{A}$ for $0<\alpha<1$, and we study spectral properties of the fractional power operators $\mathbb{A}^\alpha$. More precisely, we  establish the maximum subinterval of $(0, 1)$ where $\alpha$ is taken such that the  $-\mathbb{A}^\alpha$ is a generator, namely $-\mathbb{A}^\alpha$  generates a strongly continuous analytic semigroup  if and only if $0<\alpha< \alpha^*$ with $\alpha^*= \dfrac{3\pi}{4\pi+2\omega}$ and it generates strongly continuous semigroup  if $\alpha^*= \dfrac{3\pi}{4\pi+2\omega}$. 
  Finally, Section \ref{Sec3} provides a result of existence and uniqueness of the fractional partial differential equations \eqref{Eq2int}-\eqref{IC2int}  associated with the fractional operator $\mathbb{A}^\alpha$.  We give an  explicit representation formula of the solution of the Cauchy problem of third order  \eqref{Eq2int}-\eqref{IC2int}.

\section{The linearized operator  and its fractional powers}\label{Sec2}
Assume that $A$ be m-$\omega$-accretive operator, $0 \leq \omega\leq  \pi/2$. Let consider the following phase space 
\[
X=\mathcal{H}^{\frac{2}{3}}\times \mathcal{H}^{\frac{1}{3}}\times \mathcal{H}
\] 
which is a Banach space equipped with the norm given by
\[
\|\cdot\|_X^2=\|\cdot\|^2_{\mathcal{H}^{\frac{2}{3}}}+\|\cdot\|^2_{\mathcal{H}^{\frac{1}{3}}}+\|\cdot\|^2_{\mathcal{H}}.
\]

By letting $v=\frac{du}{dt}$, $w=\frac{d^2u}{dt^2}$, we can write the problem \eqref{Eq1}-\eqref{IC} as a Cauchy problem on $X$,   
\begin{equation}\label{Linprob}
	\begin{cases}
		\dfrac{d {\bf u}}{dt}+ \mathbb{A}{\bf u} = 0,\ t> 0,\\
		{\bf u}(0)={\bf u}_0,
	\end{cases}
\end{equation}
where the unbounded linear operator $\mathbb{A} : \mathcal{D}(\mathbb{A})\subset X \to X$ is defined by
\begin{equation}\label{LO}
	\mathbb{A} = \begin{bmatrix}
		0 & -I & 0 \\
		0 & 0 & -I \\ 
		A & 0 & 0
	\end{bmatrix}
\end{equation}
and
\begin{equation}\label{DLO}
	\mathcal{D}(\mathbb{A})=\mathcal{H}^1\times \mathcal{H}^{\frac{2}{3}}\times \mathcal{H}^{\frac{1}{3}}.
\end{equation}
Here
\[
{\bf u}=\begin{bmatrix} u\\ v \\ w\end{bmatrix}\quad \text{ and } \quad {\bf u}_0=\begin{bmatrix} u_0\\ u_1 \\ u_2 \end{bmatrix}.
\] 

\begin{remark}
	We note that $\mathbb{A}$ is not an accretive  operator on $X$. Indeed, if $u$ be a non-trivial element in $\mathcal{D}(A)$ and 
	\[
	\mathbf{u}=\left[\begin{smallmatrix}
		u\\
		0\\
		-u 
	\end{smallmatrix}\right]
	\] 
	then 
	\[
	{\rm Re}\left<\mathbb{A} \mathbf{u},\mathbf{u}\right>_X=\left<\left[\begin{smallmatrix}
		0\\
		-u\\
		Au 
	\end{smallmatrix}\right],\left[\begin{smallmatrix}
		u\\
		0\\
		-u 
	\end{smallmatrix}\right]\right>_X=	-{\rm Re} \langle Au,u\rangle_\mathcal{H}<0.
	\]
	Explicitly, this means that $-\mathbb{A}$ is not an infinitesimal generator of a strongly continuous semigroup of contractions on $Y$.  More precisely, we see  later, Remark \ref{rem:sp},  that $-\mathbb{A}$ cannot be the infinitesimal generator of a strongly continuous semigroup of any type on $X$.
\end{remark}
\subsection{Spectral properties of the operator $\mathbb{A}$}
\begin{lemma}\label{SPLO}Assume that $A$ be m-$\omega$-accretive operator, $0 \leq \omega\leq  \pi/2$. Then $\mathbb{A}$ is closed densely defined operator and  the spectrum 
	of $\mathbb{A}$ is given by
	\begin{equation}\label{spectrum} 
		\sigma(\mathbb{A})=\{\lambda \in \mathbb{C}: \lambda^3 \in \sigma(A)\}.
	\end{equation}
\end{lemma}
\begin{proof}
Firstly, we show that the operator $\mathbb{A}$ is closed. Indeed, if ${\bf u}_n={\left[\begin{smallmatrix} u_{n}\\ v_{n} \\w_{n} \end{smallmatrix}\right]}\in \mathcal{D}(\mathbb{A})$ with ${\bf u}_n\to {\bf u}={\left[\begin{smallmatrix} u\\ v \\ w \end{smallmatrix}\right]}$ in $X$ as $n\to\infty$, and $\mathbb{A} {\bf u}_n\to \varphi$ in $X$ as $n\to\infty$, where $\varphi={\left[\begin{smallmatrix} \varphi_1\\ \varphi_2 \\ \varphi_3 \end{smallmatrix}\right]}$, then 
$$
\begin{cases}
	v_{n}\to -\varphi_{1}\ \mbox{ in }\ \mathcal{H}^{\frac{2}{3}} \\
		w_{n}\to -\varphi_{2} \mbox{ in }\ \mathcal{H}^{\frac{1}{3}}\\
	A u_{n} \to \varphi_3\ \ \mbox{ in }\ \mathcal{H}\ \mbox{as}\ n\to\infty
\end{cases}
$$
and consequently, $v=-\varphi_{1}\in \mathcal{H}^{\frac{2}{3}}$, $w=-\varphi_{2}\in \mathcal{H}^{\frac{1}{3}}$. Finally, using the fact that $A$ is a closed operator, we have $u\in \mathcal{D}(A)$ and $Au=\varphi_3$; that is, ${\bf u}\in \mathcal{D}(\mathbb{A})$ and $\mathbb{A} {\bf u}=\varphi$.

Secondly, Since the inclusions $\mathcal{H}^{\alpha}\subset \mathcal{H}^{\beta}$ are dense for $\alpha>\beta\geq 0$, then we conclude that  $\mathcal{D}(\mathbb{A})$ is dense in $X$.

Finally, to prove \eqref{spectrum} it suffices to prove that 
the resolvent set of $\mathbb{A}$ is given by
\begin{equation*}
	\rho(\mathbb{A})=\{\lambda \in \mathbb{C}: \lambda^3 \in \rho(A)\}.
\end{equation*}

Suppose that $\lambda\in\mathbb{C}$ is such that $\lambda^3\in\rho(A)$. We claim that $\lambda\in\rho(\mathbb{A})$. Indeed, since $-\mathbb{A}$ is  closed operator, we only need to show that 
$$\lambda I-\mathbb{A}: \mathcal{D}(\mathbb{A})\subset X \to X$$ 
is bijective. We have
\begin{equation}\label{dasd4}
	\lambda I-\mathbb{A}= \begin{bmatrix}
		\lambda I & I & 0 \\
		0 & \lambda I & I\\ 
		-A & 0 & \lambda I
	\end{bmatrix},
\end{equation}

For injectivity consider ${\bf u}=\left[\begin{smallmatrix}
	u\\
	v\\
	w
\end{smallmatrix}\right] \in D(\mathbb{A})$ and $(\lambda I-\mathbb{A}){\bf u}=0$, then

\begin{equation}\label{sistinj}
	\begin{cases}
		\lambda u+v= 0,\\
		\lambda v+w= 0,\\
		-Au + \lambda w = 0.
	\end{cases}
\end{equation}
From \eqref{sistinj} we have
\begin{equation}
	(\lambda^3 I-A)u=0.
\end{equation}  
Since $\lambda^3\in \rho(A)$, we conclude that $u=0$ and consequently $\textbf{u}=0$. 

For surjectivity given    
$\varphi=\left[\begin{smallmatrix}
	\varphi_1\\
	\varphi_2\\
	\varphi_3
\end{smallmatrix}\right]\in X$  we take ${\bf u}=\left[\begin{smallmatrix}
	u\\
	v\\
w
\end{smallmatrix}\right]$ such that $(\lambda I-\mathbb{A}){\bf u}=\varphi$, that is
\begin{equation*}
	\begin{cases}
		\lambda u+v= \varphi_1,\\
		\lambda v+w= \varphi_2,\\
		-Au + \lambda w = \varphi_3.
	\end{cases}
\end{equation*}
Since $\lambda^3\in \rho(A)$, by straightforward calculus, we obtain 
\begin{equation*}
	\begin{cases}
		u=\lambda^2 (\lambda^{3}I-A)^{-1} \varphi_1-\lambda(\lambda^{3}I-A)^{-1}  \varphi_2+(\lambda^{3}I-A)^{-1} \varphi_3,\\
	 v= -A(\lambda^{3}I -A)^{-1} \varphi_1+\lambda^2(\lambda^{3}I-A)^{-1}  \varphi_2-\lambda (\lambda^{3}I-A)^{-1} \varphi_3,\\
	 w = \lambda A(\lambda^{3}I-A)^{-1} \varphi_1-A(\lambda^{3}I-A)^{-1}  \varphi_2+\lambda^2(\lambda^{3}I -A)^{-1} \varphi_3.
	\end{cases}
\end{equation*}
 Moreover, $u\in \mathcal{D}(A)$, $v\in \mathcal{D}(A^{\frac{2}{3}})$ and $w\in \mathcal{D}(A^{\frac{1}{3}})$. Then we have ${\bf u}\in \mathcal{D}(\mathbb{A})$ and 
\[
(\lambda I-\mathbb{A}){\bf u}=\varphi. 
\]
Now suppose that $\lambda\in \rho(\mathbb{A})$. If $u\in D(A)$ is such that $(\lambda^3 I-A)u=0$, taking ${\bf u}=\left[\begin{smallmatrix}
	u\\
	 v\\
	w
\end{smallmatrix}\right]\in \mathcal{D}(\mathbb{A})$ we have
\begin{equation}
	(\lambda I-\mathbb{A}){\bf u} = 0.
\end{equation}
Since $\lambda\in \rho(\mathbb{A})$, it follows that ${\bf u}=0$ and consequently $u=0$, which proves the injectivity of $\lambda^3 I-A$. Given $f\in \mathcal{H}$, consider $\varphi=\left[\begin{smallmatrix}
	0\\
	0\\
	f
\end{smallmatrix}\right] \in X$. By the surjectivity of $\lambda I - \mathbb{A}$ there exists ${\bf u}=\left[\begin{smallmatrix}
	u\\
	v\\
	w
\end{smallmatrix}\right] \in \mathcal{D}(\mathbb{A})$ such that
\begin{equation}
	(\lambda I - \mathbb{A}){\bf u}=\varphi
\end{equation}
which gives 
\[
(\lambda^3 I-A)u=f
\]
and the proof is complete. 
\end{proof}
\begin{remark}\label{rem:sp} Let $A$ be m-$\omega$-accretive operator, $0 \leq \omega\leq  \pi/2$. In the following some interesting consequences of Lemma \ref{SPLO} and its proof,
	\begin{enumerate}
		\item The resolvent set of $\mathbb{A}$ is given by
		\begin{equation}\label{resolvlambda} 
			\rho(\mathbb{A})=\{\lambda \in \mathbb{C}: \lambda^3 \in \rho(A)\},
		\end{equation} 
		and
		\begin{equation*}\label{}
			(\lambda I-\mathbb{A})^{-1}=	\begin{bmatrix}
				\lambda^2 (\lambda^{3}I-A)^{-1}&-\lambda(\lambda^{3}I-A)^{-1} &(\lambda^{3}I-A)^{-1} \\
				-A(\lambda^{3}I-A)^{-1} &\lambda^2(\lambda^{3}I-A)^{-1} &-\lambda (\lambda^{3}I-A)^{-1} \\
				\lambda A(\lambda^{3}I -A)^{-1} &-A(\lambda^{3}I-A)^{-1}  &\lambda^2(\lambda^{3}I -A)^{-1}
			\end{bmatrix} 
		\end{equation*}
		for all $\lambda\in\rho(\mathbb{A})$. 
		\item By \eqref{resolvlambda}, we have $$(-\infty , 0)\subset \rho(\mathbb{A}).$$
		\item 
		We have, for every $r\geq 0$,
		$$ re^{i\theta} \in \sigma(A) \Longleftrightarrow  r^{1/3}e^{i\frac{\theta}{3}}, r^{1/3}e^{i\frac{\theta+2\pi}{3}}, r^{1/3}e^{i\frac{\theta-2\pi}{3}} \in  \sigma(\mathbb{A}).$$
		Thus, we can see that the  spectrum of $\mathbb{A}$ is a reunion of possibly three disjointed  closed parts $\varLambda_1$, $\varLambda_2$ and $\varLambda_3$, with
		\begin{equation*}
	\varLambda_1 	\subset \overline{\mathcal{S}(\omega/3)}, \qquad 
	\varLambda_2 \subset e^{i\frac{2\pi}{3}}\overline{\mathcal{S}(\omega/3)},\qquad 	
\text{ and }\qquad \varLambda_3 \subset e^{-i\frac{2\pi}{3}}\overline{\mathcal{S}(\omega/3)}.
\end{equation*}
Consequently,
$$ \mathcal{S}(\pi-\omega/3) \cup\left\lbrace \lambda \in \mathbb{C}: \frac{\pi}{2}-\omega/3<\arg(\lambda) <\frac{\pi}{2}\right\rbrace  \cup \left\lbrace \lambda \in \mathbb{C}: -\frac{\pi}{2}<\arg(\lambda) <-\frac{\pi}{2}-\omega/3\right\rbrace\subset \rho(\mathbb{A} ).$$
See the figure below.
\item Clearly, $$0\in \rho(A) \Longleftrightarrow  0\in \rho(\mathbb{A}),$$
in this case, $\varLambda_1 \cap \varLambda_2\cap \varLambda_3=\emptyset$ and  $\sigma(\mathbb{A})=\varLambda_1 \cup \varLambda_2\cup \varLambda_3.$
\item If $A$ a self adjoint operator with a compact resolvent, then $-\mathbb{A}$ cannot be the infinitesimal generator of a strongly continuous semigroup on $X$. Indeed,   if $-\mathbb{A}$ generates a strongly continuous semigroup, it follows from  Pazy \cite[Theorem 1.2.2]{Pazy1983} that there exist  constants $\omega\geq 0$  such that
\begin{equation}\label{halfplan}
	\{\lambda \in \mathbb{C}: {\rm Re}\lambda > \omega\} \subset \rho(-\mathbb{A}). 
\end{equation}
But, we have
\[
\sigma_p(-\mathbb{A})=\{\lambda \in \mathbb{C}: \lambda^3 \in \sigma_p(-A)\},
\]
where $\sigma_p(-\mathbb{A})$ and $\sigma_p(-A)$ denote the point spectrum set of $-\mathbb{A}$ and $-A$, respectively.
Since $\sigma_p(-A)=\{-\mu_j: j\in \mathbb{N}\}$ with $\mu_j\in \sigma_p(A)$ for each $j\in \mathbb{N}$ and $\mu_j\to \infty$ as $j\to \infty$, we conclude that
\[
\sigma_p(-\mathbb{A})\cap \{\lambda \in \mathbb{C}: {\rm Re} \lambda>\omega\}\neq \emptyset
\]
See figure \ref{fig01}. This contradicts the equation \eqref{halfplan} and therefore $-\mathbb{A}$ can not be the infinitesimal generator of a strongly continuous semigroup on $X$.
	\end{enumerate}

\end{remark}
\begin{figure}[H]
	\begin{center}
		\begin{tikzpicture}
			\draw[-stealth', densely dotted] (-3.15,0) -- (3.15,0) node[below] {\ \ \ \ \ \ \ $\scriptstyle {\rm Re} (\lambda)$};
			\draw[-stealth', densely dotted ] (0,-3.15) -- (0,3.15) node[left] {\color{black}$\scriptstyle{\rm Im} (\lambda)$};
			\fill[blue!20!] (0,0) -- (3,0) arc (0:25:1cm)-- cycle;
			\fill[blue!20!] (0,0) -- (3,0) arc (0:-25:1cm)-- cycle;
			\fill[blue!20!] (0,0) -- (-2,3) arc (0:25:1cm)-- cycle;
			\fill[blue!20!] (0,0) -- (-2,3) arc (0:-43:1cm)-- cycle;
			\fill[blue!20!] (0,0) -- (-2,-3) arc (0:-25:1cm)-- cycle;
			\fill[blue!20!] (0,0) -- (-2,-3) arc (0:43:1cm)-- cycle;
			\node at (3.5,2) {{\tiny\color{blue} Sector $\overline{\mathcal{S}(\omega/3)}$ containing  $\varLambda_1$}};
			\node at (-3.5,1) {{\tiny\color{blue} Sector  $e^{i\frac{2\pi}{3}}\overline{\mathcal{S}(\omega/3)}$ containing  $\varLambda_2$}};
			\node at (-3.5,-1) {{\tiny\color{blue} Sector $e^{-i\frac{2\pi}{3}}\overline{\mathcal{S}(\omega/3)}$ containing  $\varLambda_3$}};
			
			\draw[color=blue,->, densely dotted] (-3,1.3) -- (-1.3,1.6);
			\draw[color=blue,->, densely dotted] (-3,-1.3) -- (-1.3,-1.6);
			\draw[color=blue,->, densely dotted] (3,1.8) -- (1.2, 0);
			\draw[color=blue, line width=1pt] (3.15,0.5) -- (0,0);
			\draw[color=blue, line width=1pt] (3.15,-0.5) -- (0,0);
			\draw[color=blue, line width=1pt] (-2, 3.3) -- (0,0);
			\draw[color=blue, line width=1pt] (-2.3, 2.3) -- (0,0);
			\draw[color=blue, line width=1pt] (-2, -3.3) -- (0,0);
			\draw[color=blue, line width=1pt] (-2.3, -2.3) -- (0,0);
			\node at (0.5,-0.2) {\color{blue}{\tiny $\omega$}};
		\end{tikzpicture}
	\end{center}
	\caption{Location of the spectrum of $\mathbb{A}$}\label{fig01}
\end{figure}
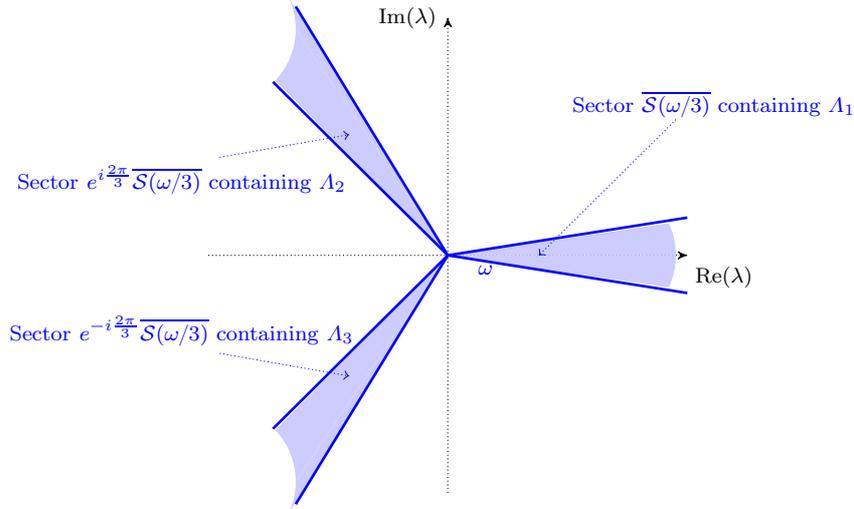

\subsection{The fractional powers of the linearized  operator}
Let $\mathbb{A} : \mathcal{D}(\mathbb{A})\subset X \to X$ be the unbounded linear operator  defined in \eqref{LO}-\eqref{DLO}. In this section we prove that $\mathbb{A}$ is an operator  of non-negative type in the sens of  \cite[ Definition 1.1.1]{MartinezS2001}. These are the operators that one can define the fractional power, for more details see Pazy  \cite[Section 2.2.6]{Pazy1983} and Amann  \cite[Section 3.4.6]{Amann1995}.

\begin{lemma}\label{PTO} Assume that $A$ be m-$\omega$-accretive operator, $0 \leq \omega\leq  \pi/2$. Then, $\mathbb{A}$  is closed, densely defined, $\left(  -\infty , 0\right) \subset \rho(\mathbb{A})$ and there exists a positive constant $M$, such that
	\begin{equation}\label{PTKR}
		\|(\lambda I+\mathbb{A})^{-1}\|_{\mathcal{B}(X)}\leqslant \frac{M}{\lambda}, \ \mbox{for all}\ \lambda> 0.
	\end{equation}
\end{lemma}
\begin{proof}
By Lemma \ref{SPLO}, the operator $\mathbb{A}$ is closed and densely defined and $\left(-\infty, 0 \right) \subset \rho(\mathbb{A})$. 

For   
$\varphi=\left[\begin{smallmatrix}
	\varphi_1\\
	\varphi_2\\
	\varphi_3
\end{smallmatrix}\right]$ , ${\bf u}=\left[\begin{smallmatrix}
	u\\
	v\\
	w
\end{smallmatrix}\right]$ in $X$ and any $\lambda>0$ we have, ${\bf u}=(\lambda I+\mathbb{A})^{-1}\varphi$ if and only if
\begin{equation*}
	\begin{cases}
		u=\lambda^2 (\lambda^{3}I+A)^{-1} \varphi_1+\lambda(\lambda^{3}I+A)^{-1}  \varphi_2+(\lambda^{3}I+A)^{-1} \varphi_3,\\
		v= -A(\lambda^{3}I +A)^{-1} \varphi_1+\lambda^2(\lambda^{3}I+A)^{-1}  \varphi_2+\lambda (\lambda^{3}I +A)^{-1} \varphi_3,\\
		w = -\lambda A(\lambda^{3}I+A)^{-1} \varphi_1-A(\lambda^{3}I+A)^{-1}  \varphi_2+\lambda^2(\lambda^{3}I+A)^{-1} \varphi_3.
	\end{cases}
\end{equation*}
 In order to verify the equation \eqref{PTKR} for $\mathbb{A}$,  it is sufficient to show that for $\|\varphi\|_X\leq 1$ there exists a constant $M_{\mathbb{A}}>0$ such that
\begin{equation}
	\|u\|_{\mathcal{H}^{\frac{2}{3}}}+\|v\|_{\mathcal{H}^{\frac{1}{3}}}+\|w\|_{\mathcal{H}} \leqslant \frac{M_{\mathbb{A}}}{\lambda}.
\end{equation}

We first observe that
\begin{equation}\nonumber
	A(\lambda^3 I + A)^{-1} = I-\lambda^3(\lambda^3 I + A)^{-1} .
\end{equation}
This and the fact that $A$ is m-accretive and $\lambda^3 \in  \rho(A)$,
\begin{equation*}
		\|A(\lambda^3 I + A)^{-1}\|_{\mathcal{B}(\mathcal{H})}\leqslant 1 + \|\lambda^3(\lambda^3 I + A)^{-1}\|_{\mathcal{B}(\mathcal{H})}\leqslant 2.\\
\end{equation*}
Now, for $x\in \mathcal{H}$, from the inequality of \cite[Theorem 6.10]{Pazy1983} we have 
\[
\begin{split}
	\|A^{\alpha}(\lambda^3 I + A)^{-1}x\|_\mathcal{H} &\leqslant C \|(\lambda^3 I + A)^{-1}x\|_\mathcal{H}^{1-\alpha}\|A(\lambda^3 I + A)^{-1}x\|_\mathcal{H}^{\alpha}\\
	&\leqslant \dfrac{2^{\alpha}C}{\lambda^{3(1-\alpha)}}\|x\|_\mathcal{H}
\end{split}
\]
for some $C> 0$. Thus,
\begin{align*}
\|u\|_{\mathcal{H}^{\frac{2}{3}}}&\leq \left\| 	\lambda^2 (\lambda^{3}I+A)^{-1}  \varphi_1\right\|_{\mathcal{H}^{\frac{2}{3}}}+\left\| \lambda(\lambda^{3}I+A)^{-1}  \varphi_2\right\|_{\mathcal{H}^{\frac{2}{3}}} + \left\| (\lambda^{3}I+A)^{-1} \varphi_3 \right\|_{\mathcal{H}^{\frac{2}{3}}}\\
&\leq \dfrac{1}{\lambda} \left\|  \varphi_1\right\|_{\mathcal{H}^{\frac{2}{3}}}+  \dfrac{2^{1/3}C}{\lambda}\left\| \varphi_2\right\|_{\mathcal{H}^{\frac{1}{3}}} +\dfrac{2^{2/3}C}{\lambda} \left\|  \varphi_3 \right\|_{\mathcal{H}}\\
&\leq \max\left\lbrace 1, 2^{1/3}C, 2^{2/3}C\right\rbrace \dfrac{1}{\lambda},  
\end{align*}
\begin{align*}
	\|v\|_{\mathcal{H}^{\frac{1}{3}}}&\leq \left\| 	A (\lambda^{3}I+A)^{-1}  \varphi_1\right\|_{\mathcal{H}^{\frac{1}{3}}}+\left\| \lambda^2(\lambda^{3}I+A)^{-1}  \varphi_2\right\|_{\mathcal{H}^{\frac{1}{3}}} + \lambda\left\| (\lambda^{3}I+A)^{-1} \varphi_3 \right\|_{\mathcal{H}^{\frac{1}{3}}}\\
	&\leq  \dfrac{2^{2/3}C}{\lambda}\left\|  \varphi_1\right\|_{\mathcal{H}^{\frac{2}{3}}}+  \dfrac{1}{\lambda}\left\| \varphi_2\right\|_{\mathcal{H}^{\frac{1}{3}}} +\dfrac{2^{1/3}C}{\lambda} \left\|  \varphi_3 \right\|_{\mathcal{H}}\\
	&\leq \max\left\lbrace 1, 2^{1/3}C, 2^{2/3}C\right\rbrace \dfrac{1}{\lambda}  
\end{align*}
and
\begin{align*}
	\|w\|_{\mathcal{H}}&\leq \left\| 	\lambda A (\lambda^{3}I+A)^{-1}  \varphi_1\right\|_{\mathcal{H}}+\left\| A(\lambda^{3}I+A)^{-1}  \varphi_2\right\|_{\mathcal{H}} + \left\|\lambda^2 (\lambda^{3}I+A)^{-1} \varphi_3 \right\|_{\mathcal{H}}\\
	&\leq  \dfrac{2^{1/3}C}{\lambda}\left\|  \varphi_1\right\|_{\mathcal{H}^{\frac{2}{3}}}+  \dfrac{2^{2/3}C}{\lambda}\left\| \varphi_2\right\|_{\mathcal{H}^{\frac{1}{3}}} +\dfrac{1}{\lambda} \left\|  \varphi_3 \right\|_{\mathcal{H}}\\
	&\leq \max\left\lbrace 1, 2^{1/3}C, 2^{2/3}C\right\rbrace \dfrac{1}{\lambda}  
\end{align*}
Hence, 
\begin{equation*}
	\|u\|_{\mathcal{H}^{\frac{2}{3}}}+\|v\|_{\mathcal{H}^{\frac{1}{3}}}+\|w\|_{\mathcal{H}} \leqslant \frac{M_{\mathbb{A}}}{\lambda},
\end{equation*}
with $M_{\mathbb{A}}=3\max\left\lbrace 1, 2^{1/3}C, 2^{2/3}C\right\rbrace$.
\end{proof}

According to Lemma \ref{PTO} and the terminology of \cite[ Definition 1.1.1]{MartinezS2001},  the operator $\mathbb{A}$ is a non-negative type, see also \cite[P. 686]{MartinezS1997}. As a particular case of this definition we see that an operator is m-accretive  if and only  it is  non-negative type  and its non-negative constant is equal to $1$, see \cite[Theorem 1.3.3]{MartinezS2001}. 

   Now, if we assume fuhrer, $0\in \rho(A)$, then there exists $0<\varepsilon \leq 1$ such that 
$${\rm Re} \left\langle Ax, x \right\rangle_\mathcal{H} \geq \varepsilon \left\|x \right\|^2_\mathcal{H}\quad \text{ for all } x\in \mathcal{D}(A). $$
Moreover, for every $\mu\geq 0$, we have
$$(\mu +\varepsilon)  \left\|x \right\|^2_\mathcal{H}\leq {\rm Re} \left\langle (\mu + A) x, x \right\rangle_\mathcal{H} \geq   \left\| (\mu + A) x \right\|_\mathcal{H} \left\| x\right\|_\mathcal{H}  \quad \text{ for all } x\in \mathcal{D}(A).$$
implies,
$$\left\|  (\mu + A)^{-1} \right\|_{\mathcal{B}(\mathcal{H})}\leq \dfrac{1}{\mu +\varepsilon}.$$
Thus, 
$$\left\|  (\mu + A)^{-1} \right\|_{\mathcal{B}(\mathcal{H})}\leq \dfrac{1}{\mu +1},  \text{ for all } \mu\geq 0.$$
Using this last inequality, Lemma \ref{PTO}, reads as follows
\begin{corollary}\label{PTOT} Assume that $A$ be m-$\omega$-accretive operator, $0 \leq \omega\leq  \pi/2$ and $0\in \rho(A)$. Then, $\mathbb{A}$  is closed, densely defined, $\left(  -\infty , 0 \right]  \subset \rho(\mathbb{A})$ and there exists a positive constant $K$, such that
	\begin{equation}\label{PTMR}
		\|(\lambda I+\mathbb{A})^{-1}\|_{\mathcal{B}(X)}\leqslant \frac{K}{\lambda+1}, \ \mbox{for all}\ \lambda\geq 0.
	\end{equation}
\end{corollary}
The operator $\mathbb{A}$ is called of positive type; \cite[ Definition 1.1.2]{MartinezS2001} and \cite[Proposition  1.1.1]{MartinezS2001}. Also  a boundedly invertible m-accretive operator is of positive type.

It is well known that,  the generator of a strongly continuous semigroup which decays exponentially   is of positive type, see Amann \cite[Page 156]{Amann1995}.  However, it possible that an  operator  of positive type do not generate a semigroup of any kind, see Remark \ref{rem:sp}-(5).

\begin{theorem}\label{thm:FPLO}Assume that $A$ be m-$\omega$-accretive operator, $0 \leq \omega\leq  \pi/2$ and $0\in \rho(A)$. Then the fractional powers $\mathbb{A}^\alpha$ of  $\mathbb{A}$ can be defined for $0\leqslant\alpha\leqslant1$ by 
	$\mathbb{A}^{\alpha}: D(\mathbb{A}^{\alpha})\subset X \to X$, 
	\begin{equation*}\label{DFPLO}
		\mathcal{D}(\mathbb{A}^{\alpha})=\mathcal{H}^{\frac{\alpha+2}{3}}\times \mathcal{H}^{\frac{\alpha+1}{3}}\times \mathcal{H}^{\frac{\alpha}{3}},
	\end{equation*}
	and
\begin{equation}\label{FPLO}
	\mathbb{A}^{\alpha}=\dfrac{1}{3} \begin{bmatrix}
		\varUpsilon_0^\alpha	A^{\frac{\alpha}{3}} &-\varUpsilon_2^\alpha A^{\frac{\alpha-1}{3}} & \varUpsilon_1^\alpha A^{\frac{\alpha-2}{3}}\\
		& & \\
		-\varUpsilon_1^\alpha A^{\frac{\alpha+1}{3}}&\varUpsilon_0^\alpha A^{\frac{\alpha}{3}} & -\varUpsilon_2^\alpha A^{\frac{\alpha-1}{3}}\\
		& & \\
		\varUpsilon_2^\alpha A^{\frac{\alpha+2}{3}} &-\varUpsilon_1^\alpha A^{\frac{\alpha+1}{3}} & \varUpsilon_0^\alpha A^{\frac{\alpha}{3}}
	\end{bmatrix}
	\end{equation}
where,
\begin{equation}\label{Varupsion}\varUpsilon_j^\alpha=2\cos\dfrac{2\pi(\alpha +j)}{3} +1, \quad \text{ for } j=0,1,2.
\end{equation}
\end{theorem}
\begin{proof} By Lemma \ref{PTO}, the  operator $\mathbb{A}$ is a non-negative type, hence we can define the fractional powers $\mathbb{A}^\alpha$ for $0<\alpha <1$ by 
	$\mathbb{A}^{\alpha}: D(\mathbb{A}^{\alpha})\subset X \to X$, and
	$$\mathbb{A}^{\alpha}=\dfrac{\sin(\alpha \pi)}{\pi}\int_0^\infty   \lambda^{\alpha-1}\mathbb{A} (\lambda I+ \mathbb{A})^{-1}d\lambda.$$
	On the other hand, 
	\begin{equation*}
(\lambda I +\mathbb{A})^{-1}=	\begin{bmatrix}
			\lambda^2 (\lambda^{3}I+A)^{-1}&\lambda(\lambda^{3}I+A)^{-1} &(\lambda^{3}I+A)^{-1} \\
			-A(\lambda^{3}I-A)^{-1} &\lambda^2(\lambda^{3}I+A)^{-1} &-\lambda (\lambda^{3}I+A)^{-1} \\
			-\lambda A(\lambda^{3}I +A)^{-1} &-A(\lambda^{3}I+A)^{-1}  &\lambda^2(\lambda^{3}I +A)^{-1}
		\end{bmatrix} 
	\end{equation*}
for all $\lambda >0$. Consequently,
\begin{equation*}
	\mathbb{A}(\lambda I +\mathbb{A})^{-1}=	\begin{bmatrix}
		A(\lambda^{3}I+A)^{-1}&-\lambda^2(\lambda^{3}I+A)^{-1} &-\lambda (\lambda^{3}I+A)^{-1} \\
		\lambda A(\lambda^{3}I-A)^{-1} &A(\lambda^{3}I+A)^{-1} &-\lambda^2 (\lambda^{3}I+A)^{-1} \\
		\lambda^2 A(\lambda^{3}I +A)^{-1} &\lambda A(\lambda^{3}I+A)^{-1}  &A(\lambda^{3}I +A)^{-1}
	\end{bmatrix} 
\end{equation*}
for all $\lambda >0$. Now we apply in each entry $[\mathbb{A}(\lambda I +\mathbb{A})^{-1}]_{ij}$ the fractional formula for $A$ given by \eqref{FRACP}, and after the change of variable $\mu=\lambda^{3}$ ($\lambda=\mu^{\frac{1}{3}}$ and $d\lambda=\frac{1}{3}\mu^{\frac{-2}{3}}d\mu$) for $i\geqslant j$, we obtain the following: 
\begin{equation}\label{A1asneesa2}
	\begin{split}
	[\mathbb{A}(\lambda I +\mathbb{A})^{-1}]_{ij}	&=\dfrac{\sin(\alpha \pi)}{\pi}\int_0^\infty   \lambda^{i-j} \lambda^{\alpha-1}A (\lambda^{3}I+A))^{-1}d\lambda\\
		&=\dfrac{1}{3} \dfrac{\sin(\alpha \pi)}{\pi}\int_0^\infty \mu^{\frac{i-j}{3}} \mu^{\frac{\alpha}{3}-1}A(\mu I+A)^{-1} d\mu\\
		&=\dfrac{(-1)^{i-j}}{3}\dfrac{\sin\Big((\frac{\alpha+i-j}{3})3 \pi\Big)}{\pi}\int_0^\infty \mu^{\frac{\alpha+i-j}{3}-1}A(\mu I+A)^{-1} d\mu,
	\end{split}
\end{equation}
Now, for any $\theta$, we have,
\begin{equation}\label{sinnt}
	\begin{split}
\sin (3\theta \pi )&= \sin (\theta \pi +2\theta \pi )=\sin (\theta \pi)\cos (2\theta \pi)+\sin (2\theta \pi)\cos (\theta \pi)\\
&=\sin (\theta \pi)(\cos (2\theta \pi)+2\cos^2 (\theta \pi))\\
&=\sin (\theta \pi)(2\cos (2\theta \pi)+1)	
	\end{split}
\end{equation}
Thus, for $\theta=\frac{\alpha+i-j}{3}$, we obtain 
\begin{equation}\label{A1asnee}
	\sin\Big((\frac{\alpha+i-j}{3})3 \pi\Big)=\Big(2\cos\Big((\frac{\alpha+i-j}{3}) \pi\Big)+1\Big)\sin\Big((\frac{\alpha+i-j}{3}) \pi\Big)
\end{equation}
and by \eqref{A1asneesa2}, \eqref{A1asnee} and $0<\alpha+i-j<3$, we have
\begin{equation}\label{A1as1214sa2}
	[\mathbb{A}(\lambda I +\mathbb{A})^{-1}]_{ij}=\dfrac{(-1)^{i-j}}{3}\Big(2\cos\Big((\frac{\alpha+i-j}{3}) \pi\Big)+1\Big)A^{\frac{\alpha+i-j}{3}}.
\end{equation}

Similarly, for $i< j$, we obtain the following:  
\begin{equation}\label{mlas335}
	\begin{split}
[\mathbb{A}(\lambda I +\mathbb{A})^{-1}]_{ij}=		&-\dfrac{\sin(\alpha \pi)}{\pi}\int_0^\infty  \lambda^{\alpha+3+i-j-1} (\lambda^{3}I+A))^{-1}d\lambda\\
		&=-\dfrac{1}{3} \dfrac{\sin(\alpha \pi)}{\pi}\int_0^\infty \mu^{\frac{\alpha+i-j}{3}}(\mu I+A)^{-1} d\mu\\
		&=\dfrac{(-1)^{i-j+4}}{3}\dfrac{\sin\Big((\frac{\alpha+i-j+3}{3})3 \pi\Big)}{\pi}\int_0^\infty \mu^{\frac{\alpha+i-j+3}{3}-1}(\mu I+A)^{-1} d\mu,
	\end{split}
\end{equation}
As above from \eqref{mlas335}, \eqref{sinnt} and $0<\alpha+i-j+3<3$, we get
\begin{equation}\label{A1am92324}
	[\mathbb{A}(\lambda I +\mathbb{A})^{-1}]_{ij}=\dfrac{(-1)^{i-j}}{3}\Big(2\cos\Big((\frac{\alpha+i-j}{3}) \pi\Big)+1\Big)A^{\frac{\alpha+i-j}{n}}.
\end{equation}
and thanks to \eqref{A1as1214sa2} and  \eqref{A1am92324} we get \eqref{FPLO} for $0<\alpha <1$.

Finally, it is easy to see that the matrix representation of $\mathbb{A}^\alpha$ holds for any $0\leqslant\alpha\leqslant1$.
\end{proof}	
\begin{remark}
\begin{enumerate}
	\item If we take $\omega =0$, we obtain \cite[Theorem 2.7]{BezerraS2020}.
	\item  The assumption that $0\in \rho(A)$ and 
	therefore a whole neighborhood of zero is in $ \rho(A)$ was made mainly for 
	convenience. The Theorem  \ref{thm:FPLO} remains true even if $0\notin \rho(A)$. 
	 \item We can see that the matrix representation of $\mathbb{A}^\alpha$ is a Toeplitz matrix for any $0\leqslant\alpha\leqslant1$.
\end{enumerate}
\end{remark}

 The next result is an immediate consequence of  \cite[Theorem 3.6]{MartinezS1997}, see also \cite[Theorem 3.2]{MartinezS1997} and \cite[Theorem 3.4]{MartinezS1997}.
 \begin{corollary}\label{SFPLO}
 	\begin{equation}\label{ew34rf}
 		\sigma(\mathbb{A}^\alpha)=\{\lambda^\alpha: \lambda\in\sigma(\mathbb{A})\}=	\{\lambda^\alpha : \lambda^3 \in \sigma(A)\}
 	\end{equation}
 	for any $0<\alpha<1$.  
 \end{corollary}
	We have, for every $r\geq 0$,
$$ re^{i\theta} \in \sigma(A) \Longleftrightarrow  r^{\alpha/3}e^{i\frac{\alpha\theta}{3}}, r^{\alpha/3}e^{\alpha i\frac{\theta+2\pi}{3}}, r^{\alpha/3}e^{i\alpha\frac{\theta-2\pi}{3}} \in  \sigma(\mathbb{A}^\alpha).$$
Thus, we can see that the  spectrum of $\mathbb{A}^\alpha$ is a reunion of possibly three disjointed  closed parts $\varGamma_1$, $\varGamma_2$ and $\varGamma_3$, with
\begin{equation}\label{spa}
	\varGamma_1 	\subset \overline{\mathcal{S}(\alpha\omega/3)}, \qquad 
	\varGamma_2 \subset e^{i\alpha\frac{2\pi}{3}}\overline{\mathcal{S}(\alpha\omega/3)},\qquad 	
	\text{ and }\qquad \varGamma_3 \subset e^{-i\alpha\frac{2\pi}{3}}\overline{\mathcal{S}(\alpha\omega/3)}.
\end{equation}
Consequently,
$$ \mathcal{S}(\pi-\frac{\alpha\omega}{3}) \cup\left\lbrace \lambda \in \mathbb{C}: \frac{\pi}{2}-\frac{\alpha\omega}{3}<\arg(\lambda) <\frac{\pi}{2}\right\rbrace  \cup \left\lbrace \lambda \in \mathbb{C}: -\frac{\pi}{2}<\arg(\lambda) <-\frac{\pi}{2}-\frac{\alpha\omega}{3}\right\rbrace\subset \rho(\mathbb{A}^\alpha ).$$

  From \cite[Remark 4.6.12]{Amann1995}, we can deduce that $-\mathbb{A}^{\alpha}$ generates a strongly continuous analytic semigroup on $X$ for $0<\alpha\leq \frac{1}{2}$. In the following proposition we prove that $\alpha$ is given in term of the angle $\omega$ of the sector for which $A$ is m-$\omega$ accretive. This is a particular interest because $\alpha$ maybe strictly greater than to $1/2$. For example, if $A$ is a self adjoint operator, i.e $\omega=0$, then   $-\mathbb{A}^{\alpha}$ generates a strongly continuous analytic semigroup on $X$ for $0<\alpha < \frac{3}{4}$. The last case is proved in \cite[Theorem 2.8]{BezerraS2020}. Our main result cover this case and we give more direct proof. 
 
 \begin{proposition}\label{GHSG}Assume that $A$ be m-$\omega$-accretive operator, $0 \leq \omega\leq  \pi/2$ and $0\in \rho(A)$. Set $\alpha^*= \dfrac{3\pi}{4\pi+2\omega}$. Then $-\mathbb{A}^{\alpha}$ generates a strongly continuous  semigroup on $X$ for $\alpha=\alpha^*$ and  generates a strongly continuous analytic semigroup on $X$ for $0<\alpha< \alpha^*$.
 \end{proposition}
 \begin{proof}Firstly, $\mathbb{A}^{\alpha}$ is closed densely defined and there exists a positive constant $K$, such that
 	\begin{equation*}
 		\|(\lambda I+\mathbb{A}^{\alpha})^{-1}\|_{\mathcal{B}(X)}\leqslant \frac{K}{\lambda+1}, \ \mbox{for all}\ \lambda\geq 0.
 	\end{equation*}
 	 Form inclusions in \eqref{spa}, we can see that the two  boundary rays of the sector $e^{i\alpha\frac{2\pi}{3}}\overline{\mathcal{S}(\alpha\omega/3)}$ are the complex numbers of argument $\alpha\frac{2\pi}{3} +\frac{\alpha\omega}{3}$ and $\alpha\frac{2\pi}{3} -\frac{\alpha\omega}{3}$, respectively. To forces this sector to be in the closed right complex half plane it suffices that $\alpha\frac{2\pi}{3} +\frac{\alpha\omega}{3}\leq\frac{\pi}{2}$. That is $\alpha\leq \dfrac{3\pi}{2(2\pi+\omega)}$. The same argument for the sector $ e^{-i\alpha\frac{2\pi}{3}}\overline{\mathcal{S}(\alpha\omega/3)}$ in the opposite direction. Consequently, 
 	$$\sigma(\mathbb{A}^\alpha ) \subset\overline{\mathcal{S}(\pi/2)}\qquad \text{ for } \alpha= \dfrac{3\pi}{4\pi+2\omega} $$
 	and
 		$$\sigma(\mathbb{A}^\alpha ) \subset\overline{\mathcal{S}(\varpi)}\qquad \text{ with }\qquad \varpi<\pi/2  \qquad \text{ for } \alpha< \dfrac{3\pi}{4\pi+2\omega}  .$$
 	Therefore,
 	$$ \mathcal{S}(\pi/2) \subset \rho(-\mathbb{A}^\alpha )\qquad \text{ for } \alpha= \dfrac{3\pi}{4\pi+2\varpi}  $$
 	and
 	$$\mathcal{S}(\frac{\pi}{2} +\varpi) \subset \rho(-\mathbb{A}^\alpha ) \qquad \text{ with }\qquad \varpi<\pi/2  \qquad \text{ for } \alpha< \dfrac{3\pi}{4\pi+2\omega} .$$	
 \end{proof}
\begin{remark} 
	We include below figures \ref{fig02} and \ref{fig03} which reflects, in particular, the loss of a ``good'' sectoriality property as $\alpha\nearrow \alpha^*$. It is  remarkable to see that when $\alpha$ decreases from one there is a rotation of the two sectors located in the left half-plane given in figure \ref{fig01} to be in the right  half-plane  and at the same time the angle of sectoriality decreases. Moreover, e.g. for $\alpha^*<\alpha\leqslant1$ the figures \ref{fig01} and \ref{fig02} reflects, in particular, the loss of well-posedness of the Cauchy problem \eqref{Linprob} and hence those of \eqref{Eq1}-\eqref{IC}. 
\end{remark}
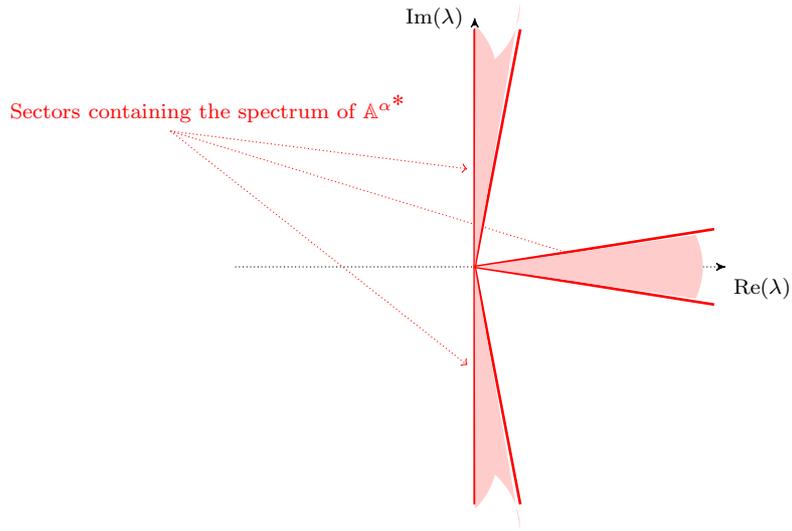
\begin{figure}[H]
	\begin{center}
		\begin{tikzpicture}
			\draw[-stealth', densely dotted] (-3.15,0) -- (3.3,0) node[below] {\ \ \ \ \ \ \ $\scriptstyle {\rm Re} (\lambda)$};
			\draw[-stealth', densely dotted ] (0,-3.15) -- (0,3.3) node[left] {\color{black}$\scriptstyle{\rm Im} (\lambda)$};
			\node at (-3.5,2.1) {{\tiny\color{red} Sectors containing the spectrum of $\mathbb{A}^{\alpha^{*}}$}};
			\draw[color=red,->, densely dotted] (-4,1.8) -- (-0.1,-1.3);
			\draw[color=red,->, densely dotted] (-4,1.8) -- (-0.1,1.3);
			\draw[color=red,->,densely dotted] (-4,1.8) -- (1.5,0.1);
			\draw[color=red, line width=1pt] (0,3.15) -- (0,0);
			\draw[color=red, line width=1pt] (0.6,3.15) -- (0,0);
			\fill[red!20!] (0,0) -- (0.3,2.5) arc (0:45:1cm)-- cycle;
			\fill[red!20!] (0,0) -- (0.6,3.5) arc (0:-50:1cm)-- cycle;
			\draw[color=red, line width=1pt] (3.15,0.5) -- (0,0);
			\draw[color=red, line width=1pt] (3.15,-0.5) -- (0,0);
			\fill[red!20!] (0,0) -- (3,0) arc (0:25:1cm)-- cycle;
			\fill[red!20!] (0,0) -- (3,0) arc (0:-25:1cm)-- cycle;
			\draw[color=red, line width=1pt] (0,-3.15) -- (0,0);
			\draw[color=red, line width=1pt] (0.6,-3.15) -- (0,0);
			\fill[red!20!] (0,0) -- (0.3,-2.5) arc (0:-45:1cm)-- cycle;
			\fill[red!20!] (0,0) -- (0.6,-3.5) arc (0:50:1cm)-- cycle;
		\end{tikzpicture}
	\end{center}
	\caption{Location of the spectrum of  $\mathbb{A}^{\alpha^*}$.}\label{fig02}
\end{figure}
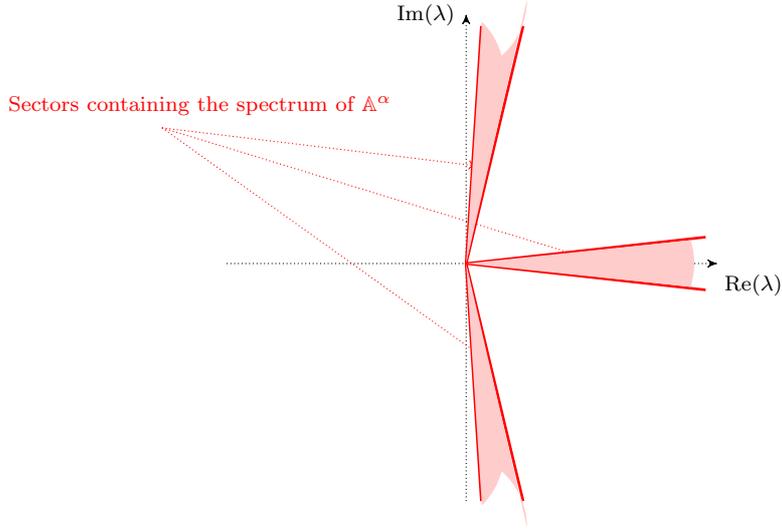
\begin{figure}[H]
	\begin{center}
		\begin{tikzpicture}
			\draw[-stealth', densely dotted] (-3.15,0) -- (3.3,0) node[below] {\ \ \ \ \ \ \ $\scriptstyle {\rm Re} (\lambda)$};
			\draw[-stealth', densely dotted ] (0,-3.15) -- (0,3.3) node[left] {\color{black}$\scriptstyle{\rm Im} (\lambda)$};
			\node at (-3.5,2.1) {{\tiny\color{red} Sectors containing the spectrum of $\mathbb{A}^{\alpha}$}};
			\draw[color=red,->, densely dotted] (-4,1.8) -- (0.3,-1.3);
			\draw[color=red,->, densely dotted] (-4,1.8) -- (0.1,1.3);
			\draw[color=red,->,densely dotted] (-4,1.8) -- (1.5,0.1);
			\draw[color=red, line width=1pt] (0.2,3.15) -- (0,0);
			\draw[color=red, line width=1pt] (0.75,3.15) -- (0,0);
			\fill[red!20!] (0,0) -- (0.5,2.5) arc (0:45:1cm)-- cycle;
			\fill[red!20!] (0,0) -- (0.8,3.5) arc (0:-50:1cm)-- cycle;
			\draw[color=red, line width=1pt] (3.15,0.35) -- (0,0);
			\draw[color=red, line width=1pt] (3.15,-0.35) -- (0,0);
			\fill[red!20!] (0,0) -- (3,0) arc (0:18:1cm)-- cycle;
			\fill[red!20!] (0,0) -- (3,0) arc (0:-18:1cm)-- cycle;
			\draw[color=red, line width=1pt] (0.2,-3.15) -- (0,0);
			\draw[color=red, line width=1pt] (0.75,-3.15) -- (0,0);
			\fill[red!20!] (0,0) -- (0.5,-2.5) arc (0:-45:1cm)-- cycle;
			\fill[red!20!] (0,0) -- (0.8,-3.5) arc (0:50:1cm)-- cycle;
		\end{tikzpicture}
	\end{center}
	\caption{Location of the spectrum of  $\mathbb{A}^{\alpha}$ with $0<\alpha<\alpha^*$.}\label{fig03}
\end{figure}

\section{A third  order Cauchy problem associated to  fractional powers of the linearized operator}\label{Sec3}

From  Theorem \ref{thm:FPLO}, we showed  that it is possible to calculate explicitly the fractional power $\mathbb{A}^\alpha$ of the operator $\mathbb{A}$ for  $0<\alpha<1$. Furthermore, $-\mathbb{A}^\alpha$ generates a strongly continuous analytic semigroup on $X$ for $0<\alpha< \alpha^*$, by Proposition \ref{GHSG}. This enables us to consider the following fist order Cauchy problem associated to the fractional power of the linearized operator,
\begin{equation}\label{fracmainprob}
	\begin{cases}
		\dfrac{d {\bf u}}{dt}+ \mathbb{A}^\alpha {\bf u} = 0,\ t> 0,\quad  0<\alpha<\alpha^*,\\
		{\bf u}(0)={\bf u}_0.
	\end{cases}
\end{equation} 
Here, $\mathbb{A}^\alpha: \mathcal{D}(\mathbb{A}^\alpha)\subset X\to X$ denotes the fractional power operator of $\mathbb{A}$. 

By Proposition \ref{GHSG} and \cite[Corollary 1.5, Chap 4.]{Pazy1983}, we can assert that for any ${\bf u}_0\in X$ the problem  \eqref{fracmainprob} admits unique solution  given by ${\bf u}(t)=e^{-t\mathbb{A}^\alpha}{\bf u}_0$ for all $0<\alpha< \alpha^*$. Also, by Proposition \ref{GHSG} and \cite[Theorem 1.3, Chap 4.]{Pazy1983}, the initial value problem \eqref{fracmainprob} has a unique solution ${\bf u}(t)=\mathcal{T}(t){\bf u}_0$, for every initial value ${\bf u}_0\in\mathcal{D}(\mathbb{A}^{\alpha^{*}})$, where $\mathcal{T}(t)$ is the strongly continuous semigroup generated by  $\mathbb{A}^{\alpha^{*}}$. 
This allows us to define the associated third  order Cauchy problem for which \eqref{fracmainprob} will be it associated linearized problem. 
\begin{equation}\label{Eq2}
u'''(t)+3\varUpsilon_0^\alpha A^{\frac{\alpha}{3}}u''(t)+ 3\varUpsilon_0^\alpha A^{\frac{2\alpha}{3}}u''(t)+A^{\alpha}u(t) = 0,\quad t> 0,
\end{equation}
with initial conditions given by
\begin{equation}\label{IC2}
	u(0)=\varphi,\ u'(0)= \psi,\ u''(0)=\xi, 
\end{equation}
where $\varUpsilon_0^\alpha=2\cos\dfrac{2\pi\alpha}{3} +1$ and $0<\alpha< \alpha^* =\dfrac{3\pi}{4\pi+2\omega}$.

\begin{theorem}Assume that $A$ be m-$\omega$-accretive operator, $0 \leq \omega\leq  \pi/2$ and $0\in \rho(A)$. Let $0<\alpha< \alpha^*$. Then for every $\varphi\in \mathcal{H}^{\frac{2}{3}}$, $\psi \in \mathcal{H}^{\frac{2-\alpha}{3}}$ and $\xi\in \mathcal{H}^{\frac{2-2\alpha}{3}}$ the Cauchy problem 
\eqref{Eq2}-\eqref{IC2} has a unique solution 
$$u\in  C^{3}((0, +\infty);\mathcal{H}) \cap C^{2}((0, +\infty); \mathcal{H}^{\frac{\alpha}{3}})\cap C^{1}((0, +\infty); \mathcal{H}^{\frac{2\alpha}{3}}) \cap C((0, +\infty); \mathcal{H}^{\alpha}), $$
and given by
\begin{equation}
	\begin{aligned}
		&	u(t)
		=e^{taA^{\frac{\alpha}{3}}}\varphi-\dfrac{a}{b-a}\Big[e^{tbA^{\frac{\alpha}{3}}}-e^{taA^{\frac{\alpha}{3}}}\Big]\varphi\\
		&+\dfrac{ab}{1+b}\Big[\dfrac{1}{b-a}e^{tbA^{\frac{\alpha}{3}}}-\dfrac{1+b}{(1+a)(b-a)}e^{taA^{\frac{\alpha}{3}}}+\dfrac{1}{1+a}e^{-tA^{\frac{\alpha}{3}}}\Big]\varphi\\
		&+\dfrac{1}{b-a}A^{-\frac{\alpha}{3}}\Big[e^{tbA^{\frac{\alpha}{3}}}-e^{taA^{\frac{\alpha}{3}}}\Big]\psi \\
		&-\dfrac{a+b}{1+b}A^{-\frac{\alpha}{3}}\Big[\dfrac{1}{b-a}e^{tbA^{\frac{\alpha}{3}}}-\dfrac{1+b}{(1+a)(b-a)}e^{taA^{\frac{\alpha}{3}}}+\dfrac{1}{1+a}e^{-tA^{\frac{\alpha}{3}}}\Big]\psi\\
		& +\dfrac{1}{1+b}A^{-\frac{2\alpha}{3}}\Big[\dfrac{1}{b-a}e^{tbA^{\frac{\alpha}{3}}}-\dfrac{1+b}{(1+a)(b-a)}e^{taA^{\frac{\alpha}{3}}}+\dfrac{1}{1+a}e^{-tA^{\frac{\alpha}{3}}}\Big]\xi,
	\end{aligned}\label{sut}
\end{equation}
with $a=e^{i\pi{\frac{3-2\alpha}{3}}}$ and $b=e^{-i\pi{\frac{3+2\alpha}{3}}}$.
\end{theorem}

\begin{proof}First, the Cauchy problem 
	\eqref{Eq2}-\eqref{IC2} is equivalent to the first order problem
	\begin{equation}\label{fracmainprob1}
			\dfrac{d }{dt}\left[\begin{smallmatrix} u\\ v \\ w\end{smallmatrix}\right]+ \mathbb{A}^\alpha \left[\begin{smallmatrix} u\\ v \\ w\end{smallmatrix}\right] = 0,\ t> 0,\quad  0<\alpha<\alpha^*,\\
	\end{equation} 
subject to the initial conditions
	\begin{equation*}
		\begin{cases}
			u(0)= \varphi \in \mathcal{H}^{\frac{2}{3}}\\
			
			v(0)=\dfrac{1}{(\varUpsilon_1^\alpha)^3-(\varUpsilon_0^\alpha)^3}\Big[\varUpsilon_1^\alpha(\varUpsilon_2^\alpha-\varUpsilon_0^\alpha)A^{\frac{1}{3}}\varphi -(\varUpsilon_2^\alpha+3\varUpsilon_0^\alpha\varUpsilon_1^\alpha)A^{\frac{1-\alpha}{3}}\psi-\varUpsilon_1^\alpha A^{\frac{1-2\alpha}{3}}\xi\Big]\in \mathcal{H}^{\frac{1}{3}},\\
			w(0) =\dfrac{1}{(\varUpsilon_2^\alpha)^3-(\varUpsilon_1^\alpha)^3}\Big[\varUpsilon_2^\alpha(\varUpsilon_0^\alpha-\varUpsilon_1^\alpha)A^{\frac{2}{3}}\varphi +(\varUpsilon_1^\alpha+3\varUpsilon_0^\alpha\varUpsilon_2^\alpha)A^{\frac{2-\alpha}{3}}\psi-\varUpsilon_2^\alpha A^{\frac{1-2\alpha}{3}}\xi\Big]\in \mathcal{H},
		\end{cases}
	\end{equation*}
where $\mathbb{A}^\alpha$ is given by \eqref{FPLO} and $\varUpsilon_j^\alpha$ are defined by \eqref{Varupsion}   $0<\alpha< \alpha^*$.
	Obviously, \eqref{fracmainprob1} is equivalent to the following system
	\begin{equation}\label{sysequ}
		\begin{cases}
			u'(t)+\varUpsilon_0^\alpha A^{\frac{\alpha }{3}} u(t)-\varUpsilon_2^\alpha A^{\frac{\alpha-1 }{3}} v(t)+\varUpsilon_1^\alpha A^{\frac{\alpha-2 }{3}} w(t)=0\\
			v'(t)-\varUpsilon_1^\alpha A^{\frac{\alpha+1 }{3}} u(t)-\varUpsilon_0^\alpha A^{\frac{\alpha}{3}} v(t)-\varUpsilon_2^\alpha A^{\frac{\alpha-1 }{3}} w(t)=0\\
				w'(t)-\varUpsilon_2^\alpha A^{\frac{\alpha+2 }{3}} u(t)-\varUpsilon_1^\alpha A^{\frac{\alpha+1}{3}} v(t)+\varUpsilon_0^\alpha A^{\frac{\alpha }{3}} w(t)=0.
		\end{cases}
	\end{equation}
By straightforward calculus, we obtain \eqref{Eq2} subject to \eqref{IC2}.  The existence and the uniqueness  result follows from  Proposition \ref{GHSG}  and \cite[Corollary 1.5, Chap 4.]{Pazy1983}. Also, we have
$$
\left[\begin{smallmatrix} u(t)\\ v(t) \\ w(t)\end{smallmatrix}\right]\in \mathcal{D} (\mathbb{A}^\alpha) =\mathcal{H}^{\frac{\alpha+2}{3}}\times \mathcal{H}^{\frac{\alpha+1}{3}}\times \mathcal{H}^{\frac{\alpha}{3}},\quad t\geq 0,\quad  0<\alpha<\alpha^*.
$$
Consequently, by \eqref{sysequ}, we get
$$
u(t)\in \mathcal{H}^{\frac{\alpha+2}{3}}, u'(t)\in \mathcal{H}^{\frac{\alpha}{3}}, u''(t)\in \mathcal{H}^{\frac{2-\alpha}{3}}, \quad t\geq 0,\quad  0<\alpha<\alpha^*.
$$
On the other hand, 	\eqref{Eq2} it can be written as the equivalent  Cauchy problem for
	the system of first order differential equations
	\begin{equation}
		\begin{gathered}
			u'(t)-aBu(t)=f(t), \\
			f'(t)-bBf(t)=g(t) , \\
			g'(t)+Bg(t)=0,\quad t> 0,
		\end{gathered}   \label{*}
	\end{equation}
where $B=A^{\frac{\alpha }{3}}$,	$a=e^{i\pi{\frac{3-2\alpha}{3}}}$ and $b=e^{-i\pi{\frac{3+2\alpha}{3}}}$. Since the operators $aB$, $bB$ and $-B$ are generators of  strongly continuous analytic semigroups, then by integrating these equations, we can write
	\begin{equation}
		\begin{gathered}
			g(t)=e^{-tB}g(0), \\
			f(t)=e^{tbB}f(0)+\int_{0}^te^{(
				t-s) bB}g( s) ds, \\
			u(t)=e^{taB}u(0)+\int_{0}^te^{(
				t-s) aB}f( s) ds.
		\end{gathered}  \label{**}
	\end{equation}
	Applying system of equations \eqref{*}, we obtain
	\begin{gather*}
		f(0)=u'(0)-aBu(0), \\
		g(0)=v'(0)-bBv(0)=u''(0)-(a+b)Bu'(0)+abB^2u(0).
	\end{gather*}
The fact that $u(0)\in \mathcal{H}^{\frac{2}{3}}\subset \mathcal{H}^{\frac{\alpha}{3}}$,  $u'(0) \in \mathcal{H}^{\frac{2-\alpha}{3}}\subset \mathcal{H}^{\frac{\alpha}{3}}$ and $u''(0)\in \mathcal{H}^{\frac{2-\alpha}{3}}\subset \mathcal{H}^{\frac{2\alpha}{3}}$, 
then, we have $g(0)\in \mathcal{H}^{\frac{2\alpha}{3}}$, $f(0)\in  \mathcal{H}^{\frac{\alpha}{3}}$ and
	\begin{equation}
		g(t)=e^{-tB}[ u''(0)-(a+b)Bu'(0)+abB^2u(0)].  \label{E4a}
	\end{equation}
	Using formulas \eqref{**}, \eqref{E4a}, we obtain
	\begin{equation} \label{E4b}
		\begin{aligned}
			f(t)&= e^{tbB}(u'(0)-aBu(0))\\
			&\quad +\frac{1}{1+b}B^{-1}\Big( e^{t bB}-e^{-tB}\Big) ( u''(0)-(a+b)Bu'(0)+abB^2u(0)).	
		\end{aligned}
	\end{equation}
	Using formulas \eqref{**}, \eqref{E4b}, we obtain by integration,
	
	\begin{equation}
		\begin{aligned}
			&u(t)	=e^{taB}\varphi+\dfrac{1}{b-a}B^{-1}\Big[e^{tbB}-e^{taB}\Big](\psi -aB\varphi)\\
			&\dfrac{1}{1+b}B^{-2}\Big[\dfrac{1}{b-a}e^{tbB}-\dfrac{1+b}{(1+a)(b-a)}e^{taB}+\dfrac{1}{1+a}e^{-tB}\Big](\xi-(a+b)B\psi+abB\varphi),
		\end{aligned}\label{sutt}
	\end{equation} 
Now, the formula \eqref{sut} holds after some manipulations.
\end{proof}



\end{document}